\documentclass[a4paper,12pt]{article}
\usepackage{amsfonts}
\usepackage{mathrsfs}
\usepackage{algorithm, algpseudocode}
\usepackage{amsmath,amssymb,amsthm,latexsym,amsfonts}
\usepackage{pstricks}
\usepackage{enumerate}

\usepackage{epsfig}
\usepackage{epstopdf}

\usepackage{graphicx}
\usepackage{color}
\usepackage{ifpdf}

\usepackage{caption}

\begin{document}

\newtheorem{lem}{Lemma}
\newtheorem{thm}{Theorem}
\newtheorem{cor}{Corollary}
\newtheorem{exa}{Example}
\newtheorem{con}{Conjecture}
\newtheorem{rem}{Remark}
\newtheorem{obs}{Observation}
\newtheorem{definition}{Definition}
\newtheorem{pro}{Proposition}
\theoremstyle{plain}
\newcommand{\D}{\displaystyle}
\newcommand{\DF}[2]{\D\frac{#1}{#2}}

\renewcommand{\figurename}{{\bf Fig}}
\captionsetup{labelfont=bf}

\title{\bf \Large 3-Rainbow index and forbidden subgraphs\footnote{Supported by NSFC No.11371205 and 11531011, and PCSIRT.}}

\author{{\small Wenjing Li, Xueliang Li, Jingshu Zhang}\\
      {\small Center for Combinatorics and LPMC}\\
      {\small Nankai University, Tianjin 300071, China}\\
       {\small liwenjing610@mail.nankai.edu.cn; lxl@nankai.edu.cn;
       jszhang@mail.nankai.edu.cn}
       }
\date{}

\maketitle
\begin{abstract}
A tree in an edge-colored connected graph $G$ is called \emph{a rainbow tree} if no
two edges of it are assigned the same color. For a vertex subset $S\subseteq V(G)$,
a tree is called an \emph{$S$-tree} if it connects $S$ in $G$.
A \emph{$k$-rainbow coloring} of $G$ is
an edge-coloring of $G$ having the property that
for every set $S$ of $k$ vertices of $G$, there exists a rainbow $S$-tree in $G$.
The minimum number of colors that are needed in a $k$-rainbow coloring of $G$
is the \emph{$k$-rainbow index} of $G$, denoted by $rx_k(G)$.
The \emph{Steiner distance $d(S)$} of a set $S$ of vertices of $G$ is
the minimum size of an $S$-tree $T$.
The \emph{$k$-Steiner diameter $sdiam_k(G)$} of $G$
is defined as the maximum Steiner distance of $S$ among all sets $S$ with $k$ vertices of $G$.
In this paper, we focus on the 3-rainbow index of graphs and find all finite families $\mathcal{F}$ of connected graphs,
for which there is a constant $C_\mathcal{F}$ such that,
for every connected $\mathcal{F}$-free graph $G$, $rx_3(G)\leq sdiam_3(G)+C_\mathcal{F}$.
\\[2mm]


\end{abstract}

\section{Introduction}

All graphs considered in this paper are simple, finite, undirected and connected.
We follow the terminology and notation of Bondy and Murty \cite{Bondy} for those not
defined here.

Let $G$ be a nontrivial connected graph with an \emph{edge-coloring c}
$:E(G)\rightarrow \{1,2,\dots,t\}$, $t\in \mathbb{N}$, where adjacent edges may
be colored with the same color. A path in $G$ is called \emph{a rainbow path} if no two edges
of the path are colored with the same color. The graph $G$ is called \emph{rainbow connected}
if for any two distinct vertices of $G$, there is a rainbow path connecting them.
For a connected graph $G$, the \emph{rainbow connection number} of $G$, denoted by $rc(G)$, is defined as the
minimum number of colors that are needed to make $G$ rainbow connected.
These concepts
were first introduced by Chartrand et al. in~\cite{Char1} and have been well-studied since then.
For further details, we refer the reader to a survey paper \cite{LSS} and a book~\cite{Li}.

In~\cite{Char2}, Chartrand et al. generalized the concept of rainbow path to rainbow tree.
A tree in an edge-colored graph $G$
is called \emph{a rainbow tree} if no two edges of it are assigned the same color.
For a vertex subset $S\subseteq V(G)$,
a tree is called an \emph{$S$-tree} if it connects $S$ in $G$.
Let $G$ be a connected graph of order $n$. For a fixed integer $k$ with $2\leq k\leq n$,
a \emph{$k$-rainbow coloring} of $G$ is
an edge-coloring of $G$ having the property that
for every $k$-subset $S$ of $G$, there exists a rainbow $S$-tree in $G$, and
in this case, the graph $G$ is called \emph{$k$-rainbow connected}.
The minimum number of colors that are needed in a $k$-rainbow coloring of $G$
is the \emph{$k$-rainbow index} of $G$, denoted by $rx_k(G)$.
Clearly, $rx_2(G)$ is just the rainbow connection number $rc(G)$ of $G$.
In the sequel, we assume that $k\geq 3$. It is easy to see that
$rx_2(G)\leq rx_3(G)\leq\dots\leq rx_n(G)$.
Recently, some results on the $k$-rainbow index have been published, especially on the 3-rainbow index.
We refer to~\cite{Cai, Chen} for more details.

The \emph{Steiner distance $d(S)$} of a set $S$ of vertices in $G$ is
the minimum size of a tree in $G$ containing $S$.
Such a tree is called a \emph{Steiner $S$-tree} or simply a \emph{Steiner tree}.
The \emph{$k$-Steiner diameter $sdiam_k(G)$} of $G$
is defined as the maximum Steiner distance of $S$ among all $k$-subsets $S$ of $G$.
Then the following observation
is immediate.

\begin{obs}\cite{Char2}\label{obs1}
For every connected graph $G$ of order $n\geq3$ and each integer
$k$ with $3\leq k\leq n$,
$$k-1\leq sdiam_k(G)\leq rx_k(G)\leq n-1.$$
\end{obs}

The authors of \cite{Char2} showed that the $k$-rainbow index of trees can achieve the upper bound.

\begin{pro}\cite{Char2}\label{pro1}
Let $T$ be a tree of order $n\geq3$. For each integer
$k$ with $3\leq k\leq n$,
$$rx_k(T)=n-1.$$
\end{pro}

From above, we notice that for
a fixed integer $k$ with $k\geq 3$, the difference $rx_k(G)-sdiam_k(G)$ can be arbitrarily large.
In fact,
if $G$ is a star $K_{1,n}$, then we have $rx_k(G)-sdiam_k(G)=
n-k$.

They also determined the precise values for the $k$-rainbow index of the cycle $C_n$
and the 3-rainbow index of the complete graph $K_n$.

\begin{thm}\cite{Char2}\label{thm1}
For integers $k$ and $n$ with $3\leq k\leq n$,
\begin{displaymath}
rx_k(C_n) = \left\{\begin{array}{ll}
n-2  & \text{if $k=3$ and $n\geq 4$}\\
n-1  & \text{if $k=n=3$ or $4\leq k\leq n$.}
\end{array}\right.
\end{displaymath}
\end{thm}

\begin{thm}\cite{Char2}\label{thm2}
\begin{displaymath}
rx_3(K_n) = \left\{\begin{array}{ll}
2  & \text{if $3\leq n \leq 5$}\\
3  & \text{if $n\geq 6$.}
\end{array}\right.
\end{displaymath}

\end{thm}

Let $\mathcal{F}$ be a family of connected graphs. We say that a graph
$G$ is \emph{$\mathcal{F}$-free} if $G$ does not contain any induced subgraph
isomorphic to a graph from $\mathcal{F}$. Specifically, for $\mathcal{F}=\{X\}$
we say that $G$ is \emph{X-free}, for $\mathcal{F}=\{X,Y\}$
we say that $G$ is \emph{(X,Y)-free}, and for $\mathcal{F}=\{X,Y,Z\}$
we say that $G$ is \emph{(X,Y,Z)-free}. The members of $\mathcal{F}$ will be referred
as \emph{forbidden induced subgraphs} in this context.
If $\mathcal{F}=\{X_1,X_2,\dots,X_k\}$, we also refer to the graphs $X_1,X_2,\dots,X_k$ as
a \emph{forbidden $k$-tuple},
and for $|\mathcal{F}|=2$ and $3$ we also say \emph{forbidden pair} and \emph{forbidden triple}, respectively.

In~\cite{P.1}, Holub et al. considered the question: For which families $\mathcal{F}$
of connected graphs, a connected $\mathcal{F}$-free graph $G$ satisfies $rc(G)\leq diam(G)+C_\mathcal{F}$,
where $C_\mathcal{F}$ is a constant (depending on $\mathcal{F}$),
and they gave a complete answer for $|\mathcal{F}|\in\{1,2\}$ in the following two
results (where $N$ denotes the \emph{net}, a graph obtained by attaching a pendant edge
to each vertex of a triangle).

\begin{thm}\cite{P.1}\label{thm3}
Let $X$ be a connected graph. Then
there is a constant $C_X$ such that every connected $X$-free graph $G$
satisfies $rc(G)\leq diam(G)+C_X$,
if and only if $X=P_3$.
\end{thm}

\begin{thm}\cite{P.1}\label{thm4}
Let $X, Y$ be connected graphs such that $X,Y \neq P_3$. Then
there is a constant $C_{XY}$ such that every connected $(X,Y)$-free graph $G$
satisfies $rc(G)\leq diam(G)+C_{XY}$,
if and only if (up to symmetry) either $X=K_{1,r} \ (r\geq4)$ and $Y=P_4$,
or $X=K_{1,3}$ and $Y$ is an induced subgraph of $N$.
\end{thm}

Let $k\geq 3$ be a positive integer.
From Observation \ref{obs1}, we know that the $k$-rainbow index is lower
bounded by the $k$-Steiner diameter.
So we wonder an analogous question concerning the $k$-rainbow index
of graphs. In this paper, we will consider the following question.

\emph{For which families $\mathcal{F}$
of connected graphs, there is a constant $C_\mathcal{F}$ such that
$rx_k(G)\leq sdiam_k(G)+C_\mathcal{F}$ if a connected graph $G$ is
$\mathcal{F}$-free ?}

In general, it is very difficult to give answers to the above question,
even if one considers the case $k=4$. So, in this paper we pay our attention only
on the case $k=3$. In Sections 3, 4 and 5, we give complete answers for the
3-rainbow index when $|\mathcal{F}|=1, 2$ and $3$, respectively.
Finally, we give a complete characterization for an arbitrary finite family $\mathcal{F}$.

\section{Preliminaries}

In this section, we introduce some further terminology and notation that will be
used in the sequel. Throughout the paper, $\mathbb{N}$ denotes the set of all
positive integers.

Let $G$ be a graph. We use $V(G)$, $E(G)$, and $|G|$ to denote the vertex set,
edge set, and the order of $G$, respectively.
For $A\subseteq V(G)$,
$|A|$ denotes the number of vertices in $A$,
and $G[A]$ denotes the subgraph of $G$
induced by the vertex set $A$.
For two disjoint subsets $X$ and $Y$ of $V(G)$,
we use $E[X,Y]$ to denote the set of edges of $G$ between $X$ and $Y$.
For graphs $X$ and $G$, we write $X\subseteq G$ if $X$ is a subgraph of
$G$, $X\overset{\scriptscriptstyle{\text{IND}}}{\subseteq}G$
if $X$ is an induced subgraph of
$G$, and $X\cong G$ if $X$ is isomorphic to $G$. In an edge-colored graph $G$, we use $c(uv)$ to denote the color assigned to an edge $uv\in E(G)$.

Let $G$ be a connect graph. For $u,v\in V(G)$, a path in $G$ from $u$ to $v$ will be referred as a \emph{$(u,v)$-path}, and, whenever necessary, it will be considered with orientation from $u$ to $v$. The \emph{distance} between $u$ and $v$ in $G$, denoted by $d_G(u,v)$, is the length of a shortest $(u,v)$-path in $G$. The \emph{eccentricity} of a vertex $v$ is $ecc(v):=max_{x\in V(G)}d_G(v,x)$. The \emph{diameter} of $G$ is $diam(G):=max_{x\in V(G)}ecc(x)$, and the \emph{radius} of $G$ is $rad(G):=min_{x\in V(G)ecc(x)}$. One can easily check that $rad(G)\leq diam(G)\leq 2rad(G)$. A vertex $x$ is \emph{central} in $G$ if $ecc(x)=rad(G)$. Let $D\subseteq V(G)$ and $x\in V(G)\setminus D$. Then we call a path $P=v_0v_1\dots v_k$ is a $v$-$D$ path if $v_0=v$ and $V(P)\cap D={v_k}$,
and $d_G(v,D):=min_{w\in D}d_G(v,w)$.

For a set $S\subseteq V(G)$ and $k\in \mathbb{N}$, we use $N_G^k(S)$ to denote the \emph{neighborhood at distance k} of $S$, i.e., the set of all vertices of $G$ at distance $k$ from $S$. In the special case when $k=1$, we simply write $N_G(S)$ for $N_G^1(S)$ and if $|S|=1$ with $x\in S$, we write $N_G(x)$ for $N_G(\{x\})$.
For a set $M\subseteq V(G)$, we set $N_M(S)=N_G(S)\cap M$ and $N_M(x)=N_G(x)\cap M$.
Finally, we will also use the \emph{closed neighborhood} of a vertex $x\in V(G)$ defined by $N_G^k[x]=(\cup_{i=1}^kN_G^i(x))\cup \{x\}$.

A set $D\subseteq V(G)$ is called \emph{dominating} if every vertex in
$V(G)\setminus D$ has a neighbor in $D$.
In addition, if $G[D]$ is connected, then we call $D$
a \emph{connected dominating set}.
A \emph{clique} of a graph $G$ is a subset $Q\subseteq V(G)$ such that $G[Q]$ is complete.
A clique is \emph{maximum} if $G$ has no clique $Q'$ with $|Q'|>|Q|$.
For a graph $G$, a subset $I\subseteq V(G)$ is called an \emph{independent set} of $G$
if no two vertices of $I$ are adjacent in $G$.
An independent set is \emph{maximum} if $G$ has no independent set $I'$ with $|I'|>|I|$.

For two positive integers $a$ and $b$, the \emph{Ramsey number} $R(a,b)$ is the smallest integer $n$ such that in any two-coloring of the edges of a complete graph on $n$ vertices $K_n$ by red and blue, either there is a red $K_a$ (i.e., a complete subgraph on $a$ vertices all of whose edges are colored red) or there is a blue $K_b$. Ramsey~\cite{R} showed that $R(a,b)$ is finite for any $a$ and $b$.

Finally, we will use $P_n$ to denote the path on $n$ vertices. An edge is called a \emph{pendant edge}
if one of its end vertices has degree one.

\section{Families with one forbidden subgraph}

In this section, we characterize all possible connected graphs $X$ such that every connected
$X$-free graph $G$ satisfies $rx_3(G)\leq sdiam_3(G)+C_X$, where $C_X$ is a constant.

\begin{thm}\label{thm5}
Let $X$ be a connected graph. Then there is a constant $C_X$
such that every connected $X$-free graph $G$
satisfies $rx_3(G)\leq sdiam_3(G)+C_X$,
if and only if $X = P_3$.
\end{thm}

\begin{figure}[h,t,b,p]
\begin{center}
\scalebox{1.2}[1.2]{\includegraphics{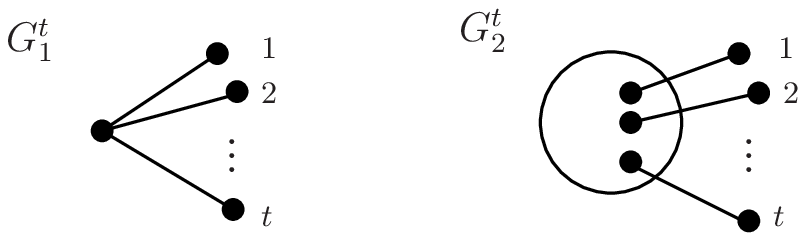}}\\[12pt]
Figure~1: The graphs $G_1^t$ and $G_2^t$.
\end{center}
\end{figure}

\begin{proof}
We have that the graph $G$ is a complete graph since $G$ is $P_3$-free.
Then from Theorem \ref{thm2},
it follows that $rx_3(G)\leq3=sdiam_3(G)+1$.

Let $t$ be an arbitrarily large integer,
set $G_1^t = K_{1,t}$,
and let $G_2^t$ denote the graph obtained by attaching a pendant edge
to each vertex of the complete graph $K_t$ (see Figure 1).
We also use $K_t^h$ to denote $G_2^t$.
Since $rx_3(G_1^t) = t$ but $sdiam_3(G_1^t) = 3$,
$X$ is an induced subgraph of $G_1^t$.
Clearly, $rx_3(G_2^t) \geq t+2$ but $sdiam_3(G_2^t) = 5$,
and $G_2^t$ is $K_{1,3}$-free. Hence, $X=K_{1,2}=P_3$. The proof is thus complete.
\end{proof}

\section{Forbidden pairs}

The following statement, which is the main result of this section,
characterizes all possible forbidden pairs
$X,Y$ for which there is a constant $C_{XY}$
such that $rx_3(G)\leq sdiam_3(G)+C_{XY}$ if $G$ is $(X,Y)$-free.
Since any $P_3$-free graph is a complete graph,
we exclude the case that one of $X,Y$
is $P_3$.

\begin{thm}\label{thm6}
Let $X,Y\neq P_3$ be a pair of connected graphs.
Then there is a constant $C_{XY}$
such that every connected $(X,Y)$-free graph $G$
satisfies $rx_3(G)\leq sdiam_3(G)+C_{XY}$,
if and only if (up to symmetry)
$X = K_{1,r},r\geq 3$ and $Y=P_4$.
\end{thm}

The proof of Theorem~\ref{thm6} will be divided into two parts.
We prove the necessity in Proposition~\ref{pro2},
and then we establish the sufficiency in Theorem~\ref{thm7}.

\begin{pro}\label{pro2}
Let $X,Y\neq P_3$ be a pair of connected graphs
for which there is a constant $C_{XY}$
such that every connected $(X,Y)$-free graph $G$
satisfies $rx_3(G)\leq sdiam_3(G)+C_{XY}$.
Then, (up to symmetry)
$X = K_{1,r},r\geq 3$ and $Y=P_4$.
\end{pro}

\begin{proof}
Let $t$ be an arbitrarily large integer,
and set $G_3^t = C_t$.
We will also use the graphs $G_1^t$ and $G_2^t$ shown in Figure 1.

Consider the graph $G_1^t$. Since $sdiam_3(G_1^t) = 3$
but $rx_3(G_1^t) =  t$, we have, up to symmetry, $X=K_{1,r},r\geq 3$.
Then we consider the graphs $G_2^t$ and $G_3^t$.
It is easy to verify that $sdiam_3(G_2^t) = 5$
but $rx_3(G_2^t) \geq t+2$, and $sdiam_3(G_3^t) = \lceil\frac{2}{3}t\rceil$
while $rx_3(G_3^t) \geq t-2 \geq \frac{3}{2}(sdiam_3(G_3^t)-1)-2$, respectively.
Clearly, $G_2^t$ and $G_3^t$ are both $K_{1,3}$-free, so neither of them contains $X$,
implying that both $G_2^t$ and $G_3^t$ contain $Y$.
Since the maximum common induced subgraph of them is $P_4$,
we get that $Y=P_4$. This completes the proof.
\end{proof}

Next, we can prove that the converse of Proposition~\ref{pro2} is true.

\begin{thm}\label{thm7}
Let $G$ be a connected $(P_4,K_{1,r})$-free graph for some
$r\geq 3$. Then $rx_3(G)\leq sdiam_3(G)+r+3$.
\end{thm}

{\noindent\bf Proof.}
Let $G$ be a connected $(P_4,K_{1,r})$-free graph $(r\geq 3)$.
Then, $sdiam_3(G)\geq 2$. For simplicity, we set $V=V(G)$.
Let $S\subseteq V$ be the maximum clique of $G$.

{\bf Claim 1:} $S$ is a dominating set.

\begin{proof}
Assume that there is a vertex $y$ at distance $2$ from $S$.
Let $yxu$ be a shortest path from $y$ to $S$, where $u\in S$.
Because $S$ is the maximum clique, there is some $v\in S$
such that $vx\notin E(G)$. Thus the path $vuxy\cong P_4$,
a contradiction. So $S$ is a dominating set.
\end{proof}

Let $X$ be the maximum independent set of $G[V\setminus S]$ and
$Y=V\setminus(S\cup X)$. Then for any vertex $y\in Y$, $y$ is adjacent to
some $x\in X$. Furthermore, for any independent set $W$ of graph
$G[Y]$, $|N_X(W)| \geq |W|$ since $X$ is maximum.

{\bf Claim 2:} There is a vertex $v\in S$ such that $v$
is adjacent to all the vertices in $X$.

\begin{proof}
Suppose that the claim fails. Let $u$ be the vertex of $S$
with the largest number of neighbors in $X$.
Set $X_1=N_X(u),\ X_2=X\setminus X_1$. Then, $X_2\neq \emptyset$
according to our assumption. Pick a vertex $w$ in $X_2$.
Then, $uw\notin E(G)$. Let $v$ be a neighbor of $w$ in $S$.
For any vertex $z$ in $X_1$, $G[{w,v,u,z}]$
can not be an induced $P_4$, so $vz$ must be an edge of $G$.
Thus, $N_X(v)\supseteq N_X(u)\cup \{w\}$, contradicting
the maximum of $u$.
\end{proof}

Let $z$ be the vertex in $S$ which is adjacent to
all the vertices of $X$.
Set $X=\{x_1,x_2,\dots,x_\ell\}$. Then, $0\leq \ell\leq r-1$
since $G$ is $K_{1,r}$-free.
Now we demonstrate a $3$-rainbow coloring of $G$
using at most $\ell+6$ colors.
Assign color $i$ to the edge $zx_i$, and $i+1$
to the edge $x_iy$ where $1\leq i \leq \ell$ and
$y\in Y$. Color $E[S,Y]$ with color $\ell+2$
and $E(G[Y])$ with color $\ell+3$.
Give a $3$-rainbow coloring of $G[S]$ using colors from
$\{\ell+4,\ell+5,\ell+6\}$. And color the remaining edges arbitrarily
(e.g., all of them with color 1).
Next, we prove that this coloring is a $3$-rainbow coloring
of $G$.

Let $W=\{u,v,w\}$ be a $3$-subset of $V$.

$(i)\ \{u,v,w\}\subseteq S\cup X$. There is a rainbow tree
containing $W$.

$(ii)\ \{u,v\}\subseteq S\cup X, w\in Y$. We
can find a rainbow tree containing an edge in $E[S,Y]$ that connects $W$.

$(iii)\ u\in S\cup X, \{v,w\}\subseteq Y$.

\ \ $a)$ If $vw\in E(G)$, then there is a rainbow tree containing the edge $vw$
that connects $W$.

\ \ $b)$ If $vw\notin E(G)$, then we have $|N_X(\{v,w\})|\geq |\{v,w\}|=2$. So
there are two vertices $x_i$ and $x_j(i\neq j)$ in $X$ adjacent to $v$
and $w$, respectively.
As $i+1\neq j+1$, so either $i+1\neq c(zu)$ or $j+1\neq c(zu)$.
Without loss of generality, we assume that $i+1\neq c(zu)$
and $s$ is a neighbor of $w$ in $S$.
Then there is a rainbow tree containing the edges $zu,uv,sw,sz$ if
$u=x_i$ or the edges $zu,zx_i,x_iv,sw,sz$ if $u\neq x_i$.

$(iv)\ \{u,v,w\}\subseteq Y$.

\ \ $a)$ If $\{uv,vw,uw\}\cap E(G)\neq \emptyset$, for example, $uv\in E(G)$,
then we have a rainbow tree containing the edges $zx_i,x_iu,uv,sw,sz$
where $x_i$ is a neighbor of $u$ in $X$ and $s$ is a neighbor of $w$
in $S$.

\ \ $b)$ If $\{uv,vw,uw\}\cap E(G)= \emptyset$, then we have $|N_X\{u,v,w\}|\geq
|\{u,v,w\}|=3$, so we can find three distinct vertices $x_i,x_j,x_k$
in $X$ such that $\{x_iu,x_jv,x_kw\}\subseteq E(G)$. We may assume that
$i<j<k$, so $k+1\notin \{i,j,k,i+1,j+1\}$ and
$k\neq i+1$.
Then there is a rainbow tree containing
the edges $zx_i,x_iu,zx_k,x_kw,sv,sz$ where $s$ is a neighbor of $v$ in $S$.

Thus the coloring is a $3$-rainbow coloring of $G$ using at most
$\ell+6\leq r+5\leq sdiam_3(G)+r+3$ colors. The proof is complete.
\ \ \ \ \ \ \ \ \ \ \ \ \ \ \ \ \ \ \ \ \ \ \ \ \ \ \ \ \ \ \ \ \ \
\ \ \ \ \ \ \ \ \ \ \ \ $\blacksquare$

Combining Proposition~\ref{pro2} and Theorem~\ref{thm7}, we can easily get Theorem~\ref{thm6}.

{\noindent\bf Remark}
When the maximum independent set of $G[V\setminus S]$, $X$, satisfies $|X|=\ell\geq4$, we just need
$\ell+5$ colors in the proof of Theorem~\ref{thm7}:
for the edges $x_\ell y$, we can color them with color $1$ instead of
color $\ell+1$. It only matters when the case $\{u,v,w\}\subseteq Y$ and $\{uv,vw,uw\}\cap E(G)= \emptyset$ happens.
Suppose $\{x_iu,x_jv,x_kw\}\subseteq E(G)$ and $i<j<k$. If $i\neq 1$ or $k\neq\ell$, it is the case in the proof above. So we turn to the case when $i=1$ and $k=l$. If $j=2$, then $j+1<4\leq\ell$ (that is why we need
the condition $\ell\geq4$). Thus, there is a rainbow tree containing
the edges $zx_j,x_jv,zx_k,x_kw,su,sz$ where $s$ is a neighbor of $u$ in $S$.
If $j\neq 2$, then there is a rainbow tree containing the edges $zx_i,x_iu,zx_j,x_jv,sw,sz$.

\section{Forbidden triples}
Now, we continue to consider more and obtain an analogous result which characterizes all forbidden triples $\mathcal{F}$ for which there is a constant $C_\mathcal{F}$
such that $G$ being $\mathcal{F}$-free implies $rx_3(G)\leq sdiam_3(G)+C_\mathcal{F}$. We exclude the cases which are covered by Theorems 5 and 6. We set:

$\mathfrak{F}_1=\{\{P_3\}\}$,

$\mathfrak{F}_2=\{\{K_{1,r},P_4\}| \ r\geq 3\}$,

$\mathfrak{F}_3=\{\{K_{1,r},Y,P_\ell\}| \ r\geq 3, Y\overset{\scriptscriptstyle{\text{IND}}}{\subseteq} K_s^h, s\geq 3,\ell >4\}$.

\begin{thm}\label{thm8}
Let $\mathcal{F}$ be a  family of connected graphs with $|\mathcal{F}|=3$ such that $\mathcal{F}\nsupseteq\mathcal{F}'$ for any $\mathcal{F}'\in \mathfrak{F}_1\cup \mathfrak{F}_2$.
Then there is a constant $C_\mathcal{F}$
such that every connected $\mathcal{F}$-free graph $G$
satisfies $rx_3(G)\leq sdiam_3(G)+C_\mathcal{F}$,
if and only if
$\mathcal{F}\in \mathfrak{F}_3$.
\end{thm}

First of all, we prove the necessity of the triples given by Theorem 8.

\begin{pro}\label{pro3}
Let $X,Y,Z\neq P_3$ be connected graphs,
$\{X,Y,Z\}\nsupseteq \mathcal{F}'$ for any $\mathcal{F}'\in \mathfrak{F}_2$,
for which there is a constant $C_{XYZ}$
such that every connected $(X,Y)$-free graph $G$
satisfies $rx_3(G)\leq sdiam_3(G)+C_{XYZ}$.
Then, (up to symmetry)
$X = K_{1,r} (r\geq 3) ,Y\overset{\scriptscriptstyle{\text{IND}}}{\subseteq}K_s^h (s\geq 3),$ and $Z=P_\ell (\ell>4)$.
\end{pro}

\begin{proof}
Let $t$ be an arbitrarily large integer, and let $G_1^t,G_2^t,G_3^t$ be the graphs defined in the proof of Proposition 2.

Firstly, we consider the graph $G_1^t$. Up to symmetry, we have $X=K_{1,r},r\geq 3$ (for the case $r=2$ is excluded by the assumptions).
Secondly, we consider the graph $G_2^t$.
The graph $G_2^t$ does not contain $X$, since it is $K_{1,3}$-free. Thus, up to symmetry, we have $G_2^t$ contains $Y$,
implying $Y\overset{\scriptscriptstyle{\text{IND}}}{\subseteq}K_s^h$ for some $s\geq 3$ (for the case $s\leq 2$ is excluded by the assumptions).
Finally, we consider the graphs $G_3^t$ and $G_3^{t+1}$.
Clearly, they are $(K_{1,3},K_3^h)$-free, so both of them contain neither $X$ nor $Y$. Hence,
we get that $Z=P_\ell$ for some $\ell> 4$ (for the case $\ell\leq 4$ is excluded by the assumptions).

This completes the proof.
\end{proof}

It is easy to observe that if $X\overset{\scriptscriptstyle{\text{IND}}}{\subseteq}
X'$, then every $(X,Y,Z)$-free graph is also $(X',Y,Z)$-free.
Thus, when proving the sufficiency of Theorem
\ref{thm8},
we will be always interested in \emph{maximal triples} of forbidden subgraphs,
i.e., triples $X,Y,Z$ such that, if replacing one of $X,Y,Z$, say $X$, with
a graph $X'\neq X$ such that $X\overset{\scriptscriptstyle{\text{IND}}}{\subseteq}
X'$, then the statement under consideration is not true for $(X',Y,Z)$-free graphs.

For every vertex $c\in V(G)$ and $i\in \mathbb{N}$, we set $\alpha_i(G,c)=$max$\{|M|\big|M\subseteq N_G^i[c], M$ is independent$\}$ and $\alpha_i^0(G,c)=$max$\{|M^0|\big|M^0\subseteq N_G^i(c), M^0$ is independent$\}$.

\begin{lem}~\cite{P.2}\label{lem2}
Let $r,s,i\in\mathbb{N}$. Then there is a constant $\alpha(r,s,i)$ such that, for every connected $(K_{1,r},K_s^h)$-free graph $G$ and for every $c\in V(G)$, $\alpha_i(G,c)< \alpha(r,s,i)$.
\end{lem}

We use the proof of Lemma~\ref{lem2} to get the following corollary concerning $\alpha_i^0(G,c)$ for each integer $i\geq 1$.

\begin{cor}\label{cor}
Let $r,s,i\in\mathbb{N}$. Then there is a constant $\alpha^0(r,s,i)$ such that, for every connected $(K_{1,r},K_s^h)$-free graph $G$ and for every $c\in V(G)$, $\alpha_i^0(G,c)< \alpha^0(r,s,i)$.
\end{cor}

\begin{proof}
For the sake of completeness, here we give a brief proof concentrating on the upper bound of $\alpha_i^0(G,c)$. We prove the corollary by induction on $i$.

For $i=1$, we have $\alpha^0(r,s,1)=r$, for otherwise $G$ contains a $K_{1,r}$ as an induced subgraph.

Let, to the contrary, $i$ be the smallest integer for which $\alpha^0(r,s,i)$ does not exist(i.e., $\alpha_i^0(G,c)$ can be arbitrarily large), choose a graph $G$ and a vertex $c\in V(G)$ such that $\alpha_i^0(G,c)\geq (r-2)R(s(2r-3),\alpha^0(r,s,i-1))$, and let $M^0=\{x_1^0,\dots,x_k^0\}\subseteq N_G^i(c)$ be an independent set in $G$ of size $\alpha_i^0(G,c)$. Obviously, $k\geq(r-2)R(s(2r-3),\alpha^0(r,s,i-1))$.
Let $Q_j$ be a shortest $(x_j^0,c)$-path in $G$, $j=1,\dots,k$. We denote $M^1\subseteq N_G^{i-1}(c)$ the set of all successors of the vertices from $M^0$ on $Q_j$, $j=1,\dots,k$, and $x_j^1$ the successor of $x_j^0$ on $Q_j$ (note that some distinct vertices in $M^0$ can have a common successor in $M^1$). Every vertex in $M^1$ has at most $r-2$ neighbors in $M^0$ since $G$ is $K_{1,r}$-free. Thus, $|M^1|\geq \frac{k}{r-2}\geq R(s(2r-3),\alpha^0(r,s,i-1))$. By the induction assumption and the definition of Ramsey number, $G[M^1]$ contains a complete subgraph $K_{s(2r-3)}$. Choose the notation such that $V(K_{s(2r-3)})=\{x_1^1,\dots,x_{s(2r-3)}^1\}$, and set $\widetilde{M^0}=N_{M^0}(K_{s(2r-3)})$. Using a matching between $K_{s(2r-3)}$ and $\widetilde{M^0}$, we can find in $G$ an induced $K_s^h$ with vertices of degree 1 in $\widetilde{M^0}$, a contradiction. For more details about finding the $K_s^h$, we refer the reader to~\cite{P.2}.
\end{proof}

Armed with Corollary 1, we can get the following important theorem.

\begin{thm}\label{thm9}
Let $r\geq 3, s\geq 3$, and $\ell>4$ be fixed integers. Then there is a constant $C(r,s,\ell)$ such that every connected $(K_{1,r},K_s^h,P_\ell)$-free graph $G$ satisfies $rx_3(G)\leq sdiam_3(G)+C(r,s,\ell)$.
\end{thm}

{\noindent\bf Proof.}
We have $diam(G)\leq \ell-2$ since $G$ is $P_\ell$-free.
Let $c$ be a central vertex of $G$, i.e., $ecc(c)=rad(G)\leq diam(G)\leq \ell-2$. And we set $S_i=\cup_{j=1}^iN_G^j[c]$ for an integer $i\geq 1$.

{\bf Claim: $rx_3(G[S_i\cup N_G^{i+1}(c)])\leq rx_3(G[S_i])+\alpha_{i+1}^0(G,c)+3$}
\begin{proof}
Let $X=\{x_1,x_2,\dots,x_{\alpha_{i+1}^0(G,c)}\}$ be the maximum independent set of $N_G^{i+1}(c)$ and
$Y=N_G^{i+1}(c)\setminus X$. Then for any vertex $y\in Y$, $y$ is adjacent to
some $x\in X$ and $s\in S$. Further more, for any independent set $W$ of graph
$G[Y]$, we have $|N_X(W)| \geq |W|$ since $X$ is maximum.

Now we demonstrate a $3$-rainbow coloring of $G[S_i\cup N_G^{i+1}(c)]$
using at most $k+\alpha_{i+1}^0(G,c)+3$ colors, where $k=rx_3(G[S_i])$.
We color the edges of $G[S_i]$ using colors $1,2,\dots,k$.
Color $E[S_i,Y]$ with color $k+1$
and $E(G[Y])$ with color $k+2$.
And assign color $j+k+2$ to the edges $E[\{x_j\},S_i]$, and $j+k+3$
to the edges $E[\{x_j\},Y]$ where $1\leq j \leq \alpha_{i+1}^0(G,c)$.
With the same argument as the proof of Theorem 7, we can prove that this coloring is a $3$-rainbow coloring
of $G[S_i\cup N_G^{i+1}(c)]$.
\end{proof}

From the proof of Corollary 1, it follows that $\alpha_1^0(G,c)\leq r-1$ and $\alpha_i^0(G,c)\leq (r-2)R(s(2r-3),\alpha^0(r,s,i-1))-1$ for each integer $i\geq 2$.
Let $\mathcal{R}(r,s)=\Sigma_{i=2}^{ecc(c)}R(s(2r-3),\alpha^0(r,s,i-1))$.
Recall that $ecc(c)\leq \ell-2$.
Repeated application of Claim gives the following:

$rx_3(G)\leq rx_3(G[N_G^{ecc(c)-1}[c]])+\alpha_{ecc(c)}^0(G,c)+3$

$\ \ \ \ \ \ \ \ \ \ \leq\dots$

$\ \ \ \ \ \ \ \ \ \ \leq rx_3({c})+\alpha_1^0(G,c)+\dots+\alpha_{ecc(c)}^0(G,c)+3ecc(c)$

$\ \ \ \ \ \ \ \ \ \ \leq 0+r+(r-2)\mathcal{R}(r,s)+2(\ell-2)$

$\ \ \ \ \ \ \ \ \ \ \leq sdiam_3(G)+(r-2)(\mathcal{R}(r,s)+1)+2(\ell-1)$.

Thus, we complete our proof. \ \ \ \ \ \ \ \ \ \ \ \ \ \ \ \  \ \ \ \ \ \ \ \ \ \ \ \
\ \ \  \ \ \ \ \ \ \ \ \ \ \ \ \ \ \ \ \ \ \ \ \ \ \ \ \ \ \ \ \ \ \ \ \ \ \ $\blacksquare$

{\noindent\bf Remark} The same as the remark in Section 4: for $i\geq1$, every time $\alpha_{i+1}^0(G,c)\geq4$ happens, we can save one color in the Claim of Theorem~\ref{thm9}.

\section{Forbidden $k$-tuples for any $k\in\mathbb{N}$}

Let $\mathcal{F}=\{X_1,X_2,X_3,\dots,X_k\}$ be a finite family of connected graphs with $k\geq4$ for which there is a constant $k_\mathcal{F}$ such that every connected $\mathcal{F}$-free graph satisfies $rx_3(G)\leq sdiam_3(G)+ C_\mathcal{F}$.
Let $t$ be an arbitrarily large integer, and let $G_1^t, G_2^t$ and $G_3^t$
be defined in Proposition 2.
For the graph $G_1^t$, Up to symmetry, we suppose that $X_1=K_r, r\geq3$ (for the case $r=2$ has been discussed in Section 3). Then, we consider the graphs $G_2^t$ and $G_3^t$.
Notice that $G_2^t$ and $G_3^t$ are both $K_{1,3}$-free, so neither of them contains $X_1$,
implying that $G_2^t$ or $G_3^t$ contains $X_i$, where $i\neq 1$.
We may assume that $X_2$ is an induced subgraph of $G_2^t$.
If $G_3^t$ contains $X_2$, then $X_2=P_4$, which is just
the case in Section 4. So we turn to the case that $G_3^t$ contains $X_i$ for some $i>2$.
Now consider the graphs $G_3^t, G_3^{t+1}, G_3^{t+2}, \dots, G_3^{t+k}$, each of which contains at least one of
$X_3, X_4, \dots, X_k$ as an induced subgraph due to
the analysis above. So it is forced that at least one of these $X_i(i\geq3)$ is isomorphic to $P_l$ for some
$l\geq 5$, which goes back to the case in Section 5.
Thus, the conclusion comes out.

\begin{thm}\label{thm10}
Let $\mathcal{F}$ be a finite family of connected graphs. Then there is a constant
$C_{\mathcal{F}}$ such that every connected $\mathcal{F}$-free graph satisfies $rx_3(G)\leq sdiam_3(G)+C_{\mathcal{F}}$,
if and only if $\mathcal{F}$ contains a subfamily $\mathcal{F'}\in \mathfrak{F}_1\cup \mathfrak{F}_2\cup
\mathfrak{F}_3$.
\end{thm}

\end{document}